\documentclass{amsart}
\usepackage[utf8]{inputenc}
\usepackage[T1]{fontenc}

\usepackage{amsmath}
\usepackage{amssymb}
\usepackage{tikz-cd}

\usepackage{amsthm,thmtools}
\usepackage[hidelinks]{hyperref}
\usepackage{cleveref}

\theoremstyle{plain}
\newtheorem{theorem}{Theorem}[section]

\newtheorem{lemma}[theorem]{Lemma}

\newtheorem{proposition}[theorem]{Proposition}

\newtheorem{corollary}[theorem]{Corollary}

\theoremstyle{definition}
\newtheorem{definition}{Definition}[section]

\theoremstyle{remark}
\newtheorem{remark}[theorem]{Remark}

\title{Non-reduced components of global nilpotent cones}
\date\today

\author{David Zhiyuan Bai}
\address{Yale University}
\email{david.bai@yale.edu}

\author{David Fang}
\address{Yale University}
\email{david.fang@yale.edu}

\begin{document}
\begin{abstract}
We determine the non-reduced components of global nilpotent cones in various cases of interest.
In particular, under the appropriate coprimality conditions, we show:

(1) the global nilpotent cone for an $L$-twisted $\operatorname{GL}_r$-Hitchin fibration associated to a curve $C$ of genus $g\ge 2$ is nowhere reduced, where $L$ is either the canonical bundle or has degree greater than $2g-2$;

(2) the global nilpotent cone for a moduli space of one-dimensional sheaves on a K3, abelian, or del Pezzo surface is nowhere reduced;

(3) suppose $\ell$ is a primitive, basepoint-free, big and nef class on a K3 surface, then a general fiber of a Beauville-Mukai system for the class $r\ell$ has primitive homology class if and only if $r=1$.

Our methods include group scheme actions on Lagrangian fibrations, a GIT-stratification of global nilpotent cones of Hitchin fibrations, and deformation to the normal cone.
\end{abstract}
\maketitle
\tableofcontents

\section{Introduction}
\subsection{The global nilpotent cone}
Let $S$ be a smooth quasiprojective surface and $\beta\in H_2(S;\mathbb Z)$ a curve class.
Under a polarization which shall remain implicit, we denote by $M=M_S(\beta,\chi)$ the moduli space of Gieseker-stable sheaves $\mathcal F$ on $S$ with proper support and
\[[\operatorname{supp}\mathcal F]=\beta,\chi(\mathcal F)=\chi,\]
where $\operatorname{supp}\mathcal F$ denotes the Fitting support of $\mathcal F$.

Denote by $|\beta|$ the Hilbert scheme of proper curves on $S$ with class $\beta$.
We then have a support map
\[h:M_S(\beta,\chi)\to|\beta|,\mathcal F\mapsto \operatorname{supp}\mathcal F.\]

Morphisms of this form, particularly when $h$ is a fibration, appear in various areas of algebraic geometry and representation theory.

If $S$ is a K3 surface, then a smooth projective $M$ is an example of a compact hyper-K\"ahler manifold of K3${}^{[n]}$-type, of which $h$ is a Lagrangian fibration.
A fibration of this form is called a Beauville-Mukai system.

When $S$ is the total space $\operatorname{Tot}_L$ of a (suitable) line bundle $L$ on a smooth curve $C$ of genus at least $2$, the choice $\beta=r[C]$, where $C$ embeds into $S$ via the zero-section, is known as the $L$-twisted ($\operatorname{GL}_r$-)Hitchin fibration.
When $L=K_C$, we simply call it the Hitchin fibration.

A more classical description of the $L$-twisted Hitchin fibration is to identify $M$ with the stable moduli space of Higgs bundles.
These are pairs $(E,\phi)$ where $E$ is a vector bundle of rank $r$ and degree $\chi+r(g-1)$, and $\phi:E\to E\otimes L$ is a morphism (known as the Higgs field).

The base $|\beta|$ of the support map is then identified with
\[B=\bigoplus_{i=1}^r H^0(C,L^{\otimes i})\]
where $rC\in |\beta|$ corresponds to $0\in B$, and the morphism sends $(E,\phi)$ to the characteristic polynomial of $\phi$.
The equivalence of the two formulations is called the BNR correspondence \cite{bnr}.

To understand the geometry of $h$, it is natural to study its fibers.
We are interested in a type of fiber known in the literature as \emph{global nilpotent cones}:
Suppose $C$ is a smooth curve on $S$ with $g(C)\ge 2$ and $\beta=r[C]$ for some $r\ge 2$, then the effective Cartier divisor $rC$ belongs to $|\beta|$, so we may define
\begin{definition}
    The fiber $h^{-1}(rC)$ is called the global nilpotent cone.
\end{definition}

The most well-known example of this is the global nilpotent cone for Hitchin fibrations.
In the classical description, it parameterizes Higgs bundles $(E,\phi)$ with nilpotent Higgs field $\phi$.

The multiplicities of the components in global nilpotent cones are wide open.
The case $r=2$ is quite well-studied:
The multiplicities have been computed for both the Hitchin fibration \cite{hitchin19} and the Beauville-Mukai system of genus $2$ \cite{hellman21}.
For the Hitchin fibration, the multiplicities of generically regular components are also known \cite{HH}.

The purpose of this article is to establish the nonexistence of reduced components in the global nilpotent cone in almost all cases where $h$ is known to be a fibration between smooth varieties.
Our results will show, in particular:
\begin{theorem}\label{intro:mainthm}
    Suppose $\gcd(r,\chi)=1$, and we are in one of the following situations:
    \begin{enumerate}
        \item[(CY1)] $S$ is a K3 or an abelian surface, and $[C]$ is primitive;
        \item[(CY2)] $S=\operatorname{Tot}_{K_C}$, and $C$ embeds as the zero-section.
        \item[(F1)] $(S,H)$ is a polarized del Pezzo surface, and $\gcd([C]\cdot H,\chi)=1$;
        \item[(F2)] $S=\operatorname{Tot}_L$ where $L\in\operatorname{Pic}^{>2g-2}(C)$, and $C$ embeds as the zero-section.
    \end{enumerate}
    Then $h:M\to |\beta|$ is a proper flat fibration between smooth varieties.
    If either $r>2$ or $r=2$ and $2\mid \deg_C N_{C/S}$, then the global nilpotent cone $h^{-1}(rC)$ is nowhere reduced.
\end{theorem}
Note that $h$ is generally not flat when $S$ is of general type, so we only consider the Calabi-Yau and Fano cases.

\begin{remark}
    For $r=2$, the additional requirement that $2\mid\deg_C N_{C/S}$ is automatic in the Calabi-Yau cases, but can fail in the Fano cases.
    Without this additional constraint, the global nilpotent cone can have reduced component(s), since $\chi(\mathcal O_{2C})$ can be odd.
\end{remark}

\subsection{Irregular components}
The components of $h^{-1}(rC)$ can be classified into the following dichotomy:
\begin{definition}
    Suppose $\mathcal F\in h^{-1}(rC)$.
    We say $\mathcal F$ is generically regular if the Fitting support of $\mathcal F$ equals the schematic support.
    Otherwise, we say $\mathcal F$ is irregular.
    A component of $h^{-1}(rC)$ is generically regular (resp.~irregular) if a generic point of it is.
\end{definition}
An example of an irregular component is the inclusion of the moduli space of stable bundles in the nilpotent cone of the Hitchin fibration.

\begin{theorem}[=\Cref{symp:irreg_nonred}]
    Suppose $S$ is symplectic and there is a smooth open $U\subset|\beta|$ containing $rC$ such that the curve family over $U$ is generically smooth and $M_U\to U$ is proper.
    Then any irregular component in $h^{-1}(rC)$ is non-reduced.
\end{theorem}
This is established using an infinitesimal duality discovered by de Cataldo, Rapagnetta and Sacc\`a \cite{dCRS}.
We show that there is a smooth subgroup scheme $P\to U$ of the relative Picard scheme of the curve class that acts on $M_U\to U$.
Then the framework established in \cite{dCRS} shows that, for every $x\in M_U$, the derivative
\[dh_x:T_{M_U,x}\to T_{U,h(x)}\]
is dual to the derivative of the action
\[d\operatorname{act}_x:\operatorname{Lie}(P_{h(x)})\to T_{M_U,x}\]
under the isomorphism $T_{M_U}\cong T^\ast_{M_U}$ induced by a holomorphic symplectic form.

When $x$ is a generic point on an irregular component, we shall show that $d\operatorname{act}_x$ fails to be injective, which means that the component cannot be reduced.

Among the cases listed in \Cref{intro:mainthm}, the hypotheses of the theorem are satisfied for (CY1) and (CY2).
On the other hand, the surfaces in (F1) and (F2) are not symplectic.

Unexpectedly, however, irregular components never appear in the Fano cases.
Let us first discuss this in the case (F2) by studying a stratification of the global nilpotent cone for twisted Hitchin fibrations.

\subsection{GIT-stratification on (twisted) Hitchin fibrations}
Let $L$ be a line bundle on $C$ such that either $L\cong K_C$ or $\deg L>2g-2$.
Consider the classical description of $M$ as $L$-twisted Higgs bundles: a pair $(E,\phi)$ where $E$ is a vector bundle with rank $r$ and degree $d=\chi+r(g-1)$, and $\phi:E\to E\otimes L$.

Like before, we assume $\gcd(r,\chi)=1$.
In this case, $M$ is smooth and $h$ is a fibration.

There is an action of $\mathbb G_m$ on $M$ given by $\lambda\cdot (E,\phi)=(E,\lambda\phi)$.
The morphism $h$ is $\mathbb G_m$-equivariant for an appropriate $\mathbb G_m$-action on $B$ with a unique fixed point at $0\in B$.
In particular, $M^{\mathbb G_m}=Z_1\sqcup\cdots\sqcup Z_N\subset h^{-1}(0)$.

Under an appropriate choice of $\mathbb G_m$-linearized line bundle, the global nilpotent cone $h^{-1}(0)$ is precisely the unstable locus of this $\mathbb G_m$-action.
This means that the general construction of \cite{HL1} gives rise to locally closed subschemes $S_1,\cdots,S_N\subset h^{-1}(0)$ giving a stratification of $h^{-1}(0)$, with each $S_m$ an affine bundle over $Z_m$.

\begin{theorem}[=\Cref{prop:twistedstrata}]\label{intro:Ltwisted_greg}
    If $\deg L>2g-2$, then $\dim S_m=\dim h^{-1}(0)$ if and only if $Z_m$ contains a generically regular point.
    In particular, all components of $h^{-1}(0)$ are generically regular in this case.
\end{theorem}

This is shown by computing the weights of the tangent bundle of $M$ with respect to this $\mathbb G_m$-action.
In the case $L=K_C$, this theory also produces another proof that irregular components are non-reduced (\Cref{cor:hitchinirregnonred}).

\subsection{Generically regular components}
Before we show that the absence of irregular components also appears in (F1), we need to discuss where multiplicities come from in generically regular components.

Consider the Hilbert scheme of points $\operatorname{Hilb}_{rC}^n$.
Just like the global nilpotent cone, this space is highly non-reduced.
Let us set $n=\chi+g(rC)-1$ and, without loss of generality, increase $\chi$ arbitrarily by twisting all stable sheaves with the polarization.
Note that this does not change the fibration nor $(\chi\bmod r)$.

There exists an open
\[\mathring{\operatorname{Hilb}}_{rC}^n\subset\operatorname{Hilb}_{rC}^n\]
such that there is a smooth morphism $\mathring{\operatorname{Hilb}}_{rC}^n\to h^{-1}(rC)$ whose image is precisely the generically regular locus (\Cref{prop:paircover}).

Hence showing the non-reducedness of the generically regular components can be reduced to showing the non-reducedness of this Hilbert scheme.
We establish the latter using \cite{luan}.
\begin{theorem}[=\Cref{greg_nonred}]\label{intro:gregnonred}
    Suppose either $r>2$ or $r=2$ and $2\mid \deg_CN_{C/S}$, then generically regular components of $h^{-1}(rC)$ are non-reduced.
\end{theorem}

\subsection{Deformation to the normal cone}
We extend the conclusion of \Cref{intro:Ltwisted_greg} to the case (F1) of del Pezzo surfaces.
Suppose $(S,H)$ is a polarized del Pezzo surface and $\gcd(rC\cdot H,\chi)=1$.
\begin{theorem}[=\Cref{prop:delpezzogreg}]\label{intro:delpezzo}
    In this case, all components in the global nilpotent cone are generically regular.
    In particular, they are non-reduced under the hypotheses of \Cref{intro:gregnonred}.
\end{theorem}
To do this, we employ a technique using deformation to the normal cone.
Similar constructions have been considered in various cases in the literature \cite{DLL97,dCMS22,zhao2025}.

For the pair $C\hookrightarrow S$, the classical construction of deformation to the normal cone produces a flat family of surfaces
\[\mathcal S\to\mathbb A^1
\]
with $\mathcal S_t\cong S$ for $t\neq 0$ and $\mathcal S_0=\operatorname{Tot}_L$ where $L=N_{C/S}$.
Note that $\deg L>2g-2$.

This actually produces a degeneration from the support fibration for the del Pezzo surface to the $L$-twisted Hitchin fibration.
We show that this degeneration has the expected properties.

Although the multiplicities in a special fiber generally say almost nothing about the multiplicities in the general fiber, we do however have an extra piece of structure: the smooth map from the Hilbert scheme $\operatorname{Hilb}_{rC}^n$.

As this map is dominant in the global nilpotent cone of the special fiber, it must be dominant in the global nilpotent cone over a nonempty open subset of $\mathbb A^1$, so we establish \Cref{intro:delpezzo}.

\subsection{Primitivity of the general fiber}
For any fibration $M\to B$ between smooth projective varieties with relative dimension $d$, one can ask whether a general fiber is primitive, \emph{i.e.}~whether its class in $H_{2d}(M;\mathbb Z)$ is indivisible.
This contains significant topological information:
Indeed, it is a necessary condition for the existence of a $C^\infty$-section.

Using the nowhere-reducedness of the global nilpotent cone, we show:
\begin{theorem}[=\Cref{prop:primfibers}]
    Suppose $S$ is a K3 surface, then the Beauville-Mukai system
    \[M_S(r[C],\chi)\to |r[C]|\]
    with $[C]$ primitive and $\gcd(r,\chi)=1$ has primitive general fibers if and only if $r=1$.
\end{theorem}
The statement that a general fiber is not primitive if $r>1$ is true more generally, namely for situations (CY1) and (F1) in \Cref{intro:mainthm} under the respective hypotheses.

\subsection*{Acknowledgement}
The authors would like to thank Junliang Shen for fostering this collaboration, and Siqing Zhang for fruitful discussions.
The authors are also grateful to Mirko Mauri for reading the first version of the preprint and suggesting a more suitable definition of Liouville structures.

D.Z.B.~would like to thank Weite Pi for helpful conversations regarding moduli of sheaves on del Pezzo surfaces.
D.F.~would like to thank Yuze Luan for communications about \cite{luan}.

D.Z.B.~was partially supported by NSF grant DMS-2301474.

\section{Picard actions on moduli spaces of sheaves}
\subsection{Group schemes and fibrations}
For our purposes, a fibration is a proper flat morphism $h:X\to B$ with connected fibers between smooth varieties.

\begin{definition}
    Let $h:X\to B$ be a fibration.
    Suppose $X$ is equipped with a holomorphic symplectic form $\sigma$, then the fibration $h:X\to B$ is called a Lagrangian fibration if the generic fiber $X_\eta$ is Lagrangian, \emph{i.e.}~$\sigma|_{X_\eta}=0$ and $2\dim X_\eta=\dim X$.
\end{definition}
This definition requires a choice of a symplectic form, which we will keep implicit.

By \cite{beauville91}, Lagrangian fibrations are generically abelian, \emph{i.e.}~the generic fiber is an abelian variety.
In particular, over a Zariski open $B^\circ\subset B$, the base-change $h^\circ:X^\circ\to B^\circ$ is a torsor over the abelian scheme that is $P^\circ=\operatorname{Aut}_{X^\circ/B^\circ}^0\to B^\circ$.

Over this open $B^\circ$, the symplectic form gives an identification $\operatorname{Lie}(P^\circ/B^\circ)\cong\Omega_{B^\circ}$ through the group action.
Let us consider the extension of this to the entire base $B$.
This is related to the notion of weakly abelian fibrations studied by Ng\^o \cite{ngo}.

\begin{definition}
    A Liouville structure for a Lagrangian fibration $h:X\to B$ is a $B$-group space $P\to B$ acting on $X\to B$ such that:
    \begin{enumerate}
        \item $P$ is quasiprojective and smooth over $B$,
        \item $P|_{B^\circ}=P^\circ$, and
        \item there is an isomorphism $\operatorname{Lie}(P/B)\cong\Omega_B$ extending $\operatorname{Lie}(P^\circ/B^\circ)\cong\Omega_{B^\circ}$.
    \end{enumerate}
\end{definition}
In particular, the group space $P\to B$ must be commutative.
\begin{remark}
    If $P$ is a subgroup space of $\operatorname{Aut}_{X/B}$, then the last property is automatic by \cite[Theorem 4.8]{kim_neron}.
\end{remark}

In \cite{dCRS}, it was observed that if a Liouville structure exists on a Lagrangian fibration, then the infinitesimal action of the group space is dual to the derivative of $h$.
More precisely:
\begin{theorem}[{\cite[Lemma 2.3.1]{dCRS}}]
    Suppose $P\to B$ is a Liouville structure for a Lagrangian fibration $h:X\to B$, then
    \begin{enumerate}
        \item the isomorphism $T_X\cong\Omega_X$ induced by the symplectic form $\sigma$ restricts to an isomorphism $T_{X/B}\cong h^\ast \Omega_B$;
        \item the composite
        \[\begin{tikzcd}
            h^\ast\operatorname{Lie}(P/B)\rar{\sim}&h^\ast\Omega_B\rar{\sim}&T_{X/B}\rar& T_X
        \end{tikzcd}\]
        is the derivative of the action.
    \end{enumerate}
    In particular, at each $x\in X$, the derivative $dh:T_{X,x}\to T_{B,h(x)}$ of $h$ is dual to the derivative $d\operatorname{act}:\operatorname{Lie}(P_{h(x)})\to T_{X,x}$ of $P_{h(x)}\to P_{h(x)}\cdot x\to X$.
\end{theorem}
\begin{corollary}\label{symp:nrd_stab}
    For any $t\in B$, let $F$ be an irreducible component of the fiber $X_t$.
    Suppose $F$ is reduced, then for a generic $x\in F$, the stabilizer $\operatorname{Stab}_{P_t}(x)$ has dimension $0$. 
\end{corollary}
\begin{proof}
    If $F$ is reduced, then a generic $x\in F$ is a smooth point of $X_t$, and therefore $h:X\to B$ is smooth at $x$.
    Hence $dh$ is surjective at $x$.
    By the preceding theorem, this is equivalent to the derivative of $P_t\to P_t\cdot x\to X$ being injective.
    As $\operatorname{Lie}(\operatorname{Stab}_{P_t}(x))$ is in its kernel, we must have $\dim\operatorname{Stab}_{P_t}(x)=0$.
\end{proof}
\subsection{Moduli of sheaves on surfaces}\label{ss:moduliofsheaves}
In this section we review facts about moduli of sheaves on surfaces.
Some examples of this will give us examples of Lagrangian fibrations, where the previous discussions apply.

Suppose $S$ is a smooth quasiprojective surface, and $\beta\in H_2(S;\mathbb Z)$ a curve class.
Denote by $|\beta|$ the Hilbert scheme of proper curves in $S$ with class $\beta$, and $\mathcal C_\beta\to|\beta|$ the associated curve family.

Let $M_S(\beta,\chi)$ be the moduli space of Gieseker-stable sheaves $\mathcal F$ on $S$ with proper support and
\[[\operatorname{supp}\mathcal F]=\beta,\chi(\mathcal F)=\chi,\]
where $\operatorname{supp}\mathcal F$ is the Fitting support of $\mathcal F$.

A priori, this space depends on an ample polarization.
We omit this specification in situations where a generic polarization is used.

There is a support map $h=h_{S,\beta,\chi}:M_S(\beta,\chi)\to |\beta|$ sending such a sheaf to its Fitting support.
\begin{definition}
    A family $f:\mathcal C\to B$ is simple if for every $B$-scheme $T$, the map $\mathcal O_T\to (f_T)_\ast\mathcal O_{\mathcal C_T}$ is an isomorphism.
\end{definition}
\begin{theorem}[{\cite[\href{https://stacks.math.columbia.edu/tag/0D2C}{Tag 0D2C}]{stacks-project}}]
    For a simple family, the Picard functor $\operatorname{Pic}_{\mathcal C/B}$ is representable by an algebraic space.
\end{theorem}
Suppose $U\subset |\beta|$ is open, and $\mathcal C_U\to U$ is a simple family.
Denote by $M_U\to U$ the base-change of $M_S(\beta,\chi)\to|\beta|$ to $U$.

Consider the open subfunctor
\[\operatorname{Pic}_{\mathcal C_U/U}^\nu\subset\operatorname{Pic}_{\mathcal C_U/U}\]
parameterizing line bundles on fibers of $\mathcal C_U/U$ whose restriction to each of the irreducible components has degree zero.
\begin{lemma}\label{lem:picnuaction}
    The tensor product gives an action of the $U$-group space $\operatorname{Pic}^\nu_{\mathcal C_U/U}\to U$ on $M_U\to U$.

    If $C$ is a smooth member of the family, then it is connected.
    The fiber of $M_U\to U$ at $C$ is precisely $\operatorname{Pic}^{\chi+g(C)-1}(C)$.
    And the action is the action of $\operatorname{Pic}^0(C)$ on it.
\end{lemma}
\begin{proof}
    The action is well-defined since tensoring with line bundles in $\operatorname{Pic}^\nu$ does not change the degree of a pure sheaf.

    A smooth member $C$ of the family is necessarily connected since it is simple.
    Pure sheaves on $C$ are locally free.
    To have Fitting support $[C]$, they have to be line bundles, which themselves are automatically stable.
\end{proof}

Now suppose $\beta=r\ell$ for some integer $r>0$ such that $|\ell|$ contains a smooth curve $C$ of genus $g\ge 2$.
Note that $rC$ may not be simple in general.
\begin{lemma}\label{lem:rCgenus}
    Suppose $\deg N_{C/S}>0$, then $rC$ is simple for all $r\ge 1$.
    Moreover, it has arithmetic genus
    \[h^1(rC,\mathcal O_{rC})=\binom{r}{2}\deg N_{C/S}+r(g-1)+1.\]
\end{lemma}
\begin{proof}
    We make use of the exact sequence
    \[\begin{tikzcd}
        0\rar& (i_C)_\ast(N_{C/S}^{\otimes -r})\rar& \mathcal O_{(r+1)C}\rar&\mathcal O_{rC}\rar&0.
    \end{tikzcd}\]
    Since $H^0(N_{C/S}^\vee)=0$, we have $H^0((r+1)C)\le H^0(rC)$.
    Inductively, this shows that $rC$ is simple for all $r\ge 1$.

    We then have
    \[\chi(\mathcal O_{(r+1)C})-\chi(\mathcal O_{rC})=\chi(N_{C/S}^{\otimes-r})=-r\deg N_{C/S}-g+1\]
    which gives the genus formula by induction.
\end{proof}
\begin{definition}
    For $C\in|\ell|$ smooth, the global nilpotent cone at $rC$ is defined as the fiber of the support map at $rC$.

    For a sheaf $\mathcal F$ in the global nilpotent cone at $rC$, we say it is generically regular if its scheme-theoretic support is precisely $rC$.
    Otherwise, we say it is irregular.
\end{definition}
Let us remark that generically regular is an open condition.
\begin{definition}
    We say a component of a global nilpotent cone is generically regular (resp.~irregular) if a generic point of it is generically regular (resp.~irregular).
\end{definition}

\subsection{Sheaves on symplectic surfaces}
Let us now specialize to the situation where $S$ is symplectic.
As before let $\beta$ be a homology class and $U\subset|\beta|$ a smooth open subset such that the curve family $\mathcal C_U$ over $U$ is generically smooth.
\begin{proposition}
    Suppose the surface $S$ is symplectic, the family $\mathcal C_U/U$ is simple, and $M_U\to U$ is proper.
    Then $M_U\to U$ is a Lagrangian fibration, for which the action of $\operatorname{Pic}^\nu_{\mathcal C_U/U}\to U$ on it is a Liouville structure.
\end{proposition}
\begin{proof}
    It is well-known that $M_U$ is smooth and there is a symplectic form on it given by Serre duality (as sheaves parameterized by $M_U$ are simple and have proper supports).

    If $L$ is a line bundle on a smooth curve in $U$, then the diagram
    \[\begin{tikzcd}
        \operatorname{Ext}^1(L,L)\times\operatorname{Ext}^1(L,L)\rar\dar&\operatorname{Ext}^2(L,L)=0\dar\\
        \operatorname{Ext}^1((i_{D})_\ast L,(i_{D})_\ast L)\times \operatorname{Ext}^1((i_{D})_\ast L,(i_{D})_\ast L)\rar&\operatorname{Ext}^2((i_{D})_\ast L,(i_{D})_\ast L)=\mathbb C
    \end{tikzcd}\]
    commutes, which shows that a generic fiber of $M_U\to U$ is Lagrangian.

    Since the curve family is simple and generically smooth, $M_U\to U$ is surjective and has connected generic fiber by \Cref{lem:picnuaction}.
    Flatness follows from \cite{matequid}.

    By Lemma \ref{lem:picnuaction}, it remains to compute the Lie algebra of $\operatorname{Pic}^\nu$.
    But this follows from the classical identifications $T_{[C]}U\cong H^0(C,N_{C/S})\cong H^0(C,\omega_C)$, where the second isomorphism follows from $S$ being symplectic.
\end{proof}
Now suppose $\beta=r[C]$ where $C$ is a smooth curve of genus $g\ge 2$ with $rC\in U$.
\begin{corollary}\label{symp:irreg_nonred}
    Suppose $S$ is symplectic and $M_U\to U$ is proper.
    Then any irregular component of the global nilpotent cone at $rC$ is non-reduced.
\end{corollary}
\begin{proof}
    By \Cref{lem:rCgenus}, $rC$ is simple and has genus at least $2$.
    Possibly after shrinking $U$ further, we may assume that $\mathcal C_U\to U$ is a simple family.
    The preceding proposition applies and gives a Liouville structure on $M_U\to U$.

    Let $\mathcal F$ be a generic sheaf in that irregular component.
    By definition, its scheme-theoretic support is $r'C$ for some $0<r'<r$.
    Write $\mathcal F=(i_{r'C})_\ast\mathcal G$.
    So for any $L\in\operatorname{Pic}(rC)$, we have
    \[((i_{r'C})_\ast\mathcal G)\otimes_{\mathcal O_{rC}} L=(i_{r'C})_\ast(\mathcal G\otimes_{\mathcal O_{r'C}} L|_{r'C}).\]
    Hence the kernel of $\operatorname{Pic}^\nu(rC)\to\operatorname{Pic}^\nu(r'C)$ is contained in the stabilizer of $\mathcal F$ under the Liouville structure described earlier in this section.

    By \Cref{symp:nrd_stab}, it suffices to show that the kernel of $\operatorname{Pic}((r+1)C)\to\operatorname{Pic}(rC)$ is positive-dimensional for all $r\ge 1$.
    But by \Cref{lem:rCgenus} again, the kernel has dimension precisely $(2r+1)(g-1)$ which is positive.
\end{proof}
Let us now give concrete examples where this can be applied.
\subsubsection{Projective surfaces}
Suppose $S$ is a projective surface with trivial canonical bundle.
Then it is symplectic.
In fact, it is either a K3 surface or an abelian surface.

Take $\ell=[C]$ for a smooth curve $C$ on $S$ with genus at least $2$.
In these cases, $|r\ell|$ is a projective bundle over an abelian variety, and the curve family is generically smooth.
The latter claim is clear for K3 surfaces ($C$ would be basepoint-free); for abelian surfaces, it follows from Lefschetz's theorem that every ample line bundle on an abelian variety has globally generated square and very ample higher powers.

Suppose $\ell$ is primitive, then $M_S(r\ell,\chi)$ is a projective holomorphic symplectic variety when $\gcd(r,\chi)=1$.
In particular, the support map is proper.
So \Cref{symp:irreg_nonred} applies, and we obtain:
\begin{corollary}\label{k3:irreg_nonred}
    Irregular components of the global nilpotent cone at $rC$ are non-reduced.
\end{corollary}

\subsubsection{The (twisted) Hitchin fibration}
Suppose $C$ is a smooth curve of genus $g\ge 2$, and $L\in\operatorname{Pic}(C)$ is either $K_C$ or has degree greater than $2g-2$.
Let
\[S=\operatorname{Tot}_L\to C\]
be the total space of $L$ over $C$.

View $C$ as a subvariety of $S$ by identifying it with the zero-section.
Let $\ell=[C]$.
A generic member of the system $|r\ell|$ of spectral curves is smooth.

Suppose $\gcd(r,\chi)=1$, then $M_S(r\ell,\chi)$ is smooth and the support map
\[M_S(r\ell,\chi)\to |r\ell|\]
is proper flat (see \cite{chaudlaumon,MSchi}).
Note that $|r\ell|$ is an affine space.
\begin{definition}
    The map $M_S(r\ell,\chi)\to |r\ell|$ is known as the $L$-twisted ($\operatorname{GL}_r$-)Hitchin fibration.
\end{definition}

When $L=K_C$, we simply call this the Hitchin fibration.
The surface $S$ is symplectic in this case, so it follows from \Cref{symp:irreg_nonred} that:
\begin{corollary}\label{cor:hitchin_using_group_scheme}
    Irregular components of the global nilpotent cone for the Hitchin fibration are non-reduced.
\end{corollary}
\begin{remark}
    One can define the Hitchin fibration for a general (connected) reductive group.
    When the twist is $K_C$, a similar Liouville structure always exists \cite{dCHM}.
    So \Cref{symp:nrd_stab} can be applied to any such cases.
\end{remark}
We will provide another proof of this soon, using the theory of weights.

\section{Moduli of Higgs bundles}\label{sec:Higgs}

\subsection{\texorpdfstring{$\mathbb{G}_m$}{Gm}-action and GIT-stratification}\label{sec:GIT}
To start, we recall some results about Bia{\l}ynicki-Birula stratifications for twisted Hitchin fibrations.

Fix a curve $C$ of genus $g\ge 2$.
Let $r>0$ and $d$ be integers with $\gcd(r,d)=1$, and let $L$ be a line bundle on $C$ with either $L=K_C$ or $\deg L >2g-2$.

We denote by $M_{r,d}^L$ the moduli space of stable $L$-twisted Higgs bundles of rank $r$ and degree $d$, that is, the space of pairs $(E,\phi)$ where $E$ is a vector bundle on $C$ of rank $r$ and degree $d$, and $\phi:E\to E\otimes L$.

This is a smooth quasiprojective variety, admitting a Hitchin fibration
\[h_L: M_{r,d}^L \rightarrow B:= \bigoplus_{i=1}^r H^0(C, L^{\otimes i}),\]
sending a pair $(E,\phi)$ to the characteristic polynomial of $\phi$.
This is the same as the support fibration on the total space of $L$ constructed before by the BNR correspondence \cite{bnr}, by setting $\chi=d+r(1-g)$.
Under this correspondence, $rC\in |\beta|$ corresponds to the point $0\in B$.

We have an action of $\mathbb{G}_m$ on $M_{r,d}^L$ by $\lambda\cdot(E, \phi)= (E,\lambda\phi)$.
Suppose $(E,\phi)$ is a fixed point; then up to normalization the $\lambda$-action decomposes $E$ into eigenspaces
\[E = E_0 \oplus\cdots\oplus E_k,\]
where the weight of the $\mathbb{G}_m$-action on $E_i$ is $-i$, and $\phi(E_i)\subset E_{i+1}\otimes L$.

Let us enumerate the components of $(M_{r,d}^L)^{\mathbb{G}_m}$ as $Z_1, \dots, Z_N$.
These $Z_m$ are smooth since $M_{r,d}^L$ is.
The goal of this section is to realize the nilpotent cone $h_L^{-1}(0)$ as a union of GIT-strata. 

First let us consider the vector bundle $\underline{\operatorname{End}}(\mathbb{E})$ on $M_{r,d}^L\times C$, where $\mathbb{E}$ is a universal (twisted) Higgs bundle.
Let $\mathcal L:= \det(\mathbf R\operatorname{pr}_{M\ast}\underline{\operatorname{End}}(\mathbb{E})[1])$.
By the arguments in \cite{Nisture}, this is relatively ample for $h_L$.

On the other hand, the map $h_L$ is $\mathbb{G}_m$-equivariant for the $\mathbb{G}_m$-action on $B$ given by $(a_1,\ldots,a_r)\mapsto (\lambda a_1,\ldots,\lambda^r a_r)$.
So each $Z_m$ is contained in $h_L^{-1}(0)$.

Let $\mathcal O_B(1)$ be a $\mathbb{G}_m$-equivariant line bundle such that the weight of $\mathcal O_B(1)|_0$ is positive.
Then for $N$ sufficiently large the weight of the $\mathbb{G}_m$-linear polarization
\[\mathcal L' := \mathcal L\otimes h_L^\ast\mathcal O_B(N)\]
on each $Z_m$ is positive.
In particular, by the Hilbert-Mumford criterion, the unstable points of $M_{r,d}^L$ are precisely those which are contracted to some $Z_m$ under the repelling $\mathbb{G}_m$-action.
Since $h_L$ is proper and $\mathbb{G}_m$-equivariant, this means that the unstable locus is precisely $h_L^{-1}(0)$. 

Now by carrying out the construction outlined in \cite[\S2.1]{HL1}, we obtain a stratification of $h^{-1}_L(0)^{red}$ into smooth locally closed subschemes $S_1, \dots, S_N$ such that:
\begin{itemize}
    \item $S_m$ is a locally trivial affine bundle over $Z_m$
    \item\label{test} For $z_m\in Z_m\subset (M_{r,d}^L)^{\mathbb{G}_m}$, we can compute the tangent and normal spaces:
    \begin{equation}\label{eq:GITtangents}
        T_{S_m, z_m} = T_{M_{r,d}^L, z_m}^{\le 0}, \qquad N_{S_m/M_{r,d}^L, z_m} = T_{M_{r,d}^L, z_m}^{>0},
    \end{equation}
    where the weights are considered with respect to the standard (contracting) $\mathbb{G}_m$-action on $M_{r,d}^L$.
\end{itemize}

When $L = K_C$, we have an additional piece of structure:
\begin{theorem}[{\cite[Proposition 3.9]{HH}}]\label{prop:lagrangianstrata}
    Let $L = K_C$; in this case $h_L$ is a Lagrangian fibration.
    Then all of the $S_m$ are Lagrangian submanifolds.
\end{theorem}
In particular, the argument above realizes the components of the ``downward flows'' in \cite{HH} as GIT-strata. 

When $\deg L > 2g-2$, we shall show the following theorem by computing the tangent spaces in \Cref{eq:GITtangents} directly.
\begin{proposition}\label{prop:twistedstrata}
    Suppose $\deg L > 2g-2$. Then $\dim S_m = \dim h_L^{-1}(0)$ if and only if the points of $Z_m$ are generically regular. 
\end{proposition}
\begin{corollary}
    All components of $h_L^{-1}(0)$ are generically regular if $\deg L>2g-2$.
\end{corollary}

\subsection{Tangent spaces and weights}\label{sec:tangents}

Let $(E, \phi)$ be a stable Higgs bundle as above. Let 
\[\mathbb{T}_{(E,\phi)}:= [\underline{\operatorname{End}}(E) \xrightarrow{[-,\phi]}\underline{\operatorname{End}}(E)\otimes L]\in D^{[-1,0]}_{\operatorname{Coh}}(C);\]
it is well known (see for example \cite{Biswas1994AnIS}) that the tangent space to $M_{r,d}^L$ at $(E,\phi)$ can be computed via the hypercohomology of $\mathbb{T}_{(E,\phi)}$:
\[T_{M_{r,d}^L, (E,\phi)} = \mathbb{H}^0(C, \mathbb{T}_{(E,\phi)}).\]
Suppose $(E,\phi)$ is additionally fixed by $\mathbb{G}_m$.
Since $(E,\phi)$ is stable (hence indecomposable), following \cite[\S3.1]{HH}, there is a unique decomposition of $E$ into weight spaces
\begin{equation}\label{eq:bundledecomposition}
    E = E_0 \oplus \cdots \oplus E_k, \qquad \operatorname{weight}(E_i) = w_0-i
\end{equation}
for some $w_0\in\mathbb Z$, with
\[0\neq \phi(E_i)\subset E_{i+1}\otimes L\]
for $0\le i\le k-1$.
Without loss of generality, we may assume $w_0=0$.

This induces a decomposition of $\mathbb{T}_{(E,\phi)}$ into eigenspaces:
\[\mathbb{T}_{(E,\phi),w}:= \left[\bigoplus_{i-j=w}\underline{\operatorname{Hom}}(E_i, E_j) \xrightarrow{[-,\phi]}\bigoplus_{i-j=w}\underline{\operatorname{Hom}}(E_i, E_{j+1})\otimes L\right] \in D_{\operatorname{Coh}}^{[-1,0]}(C).\]
Thus we get a decomposition of the tangent space into weight spaces:
\begin{equation}\label{eq:tangentdecomposition}
    (T_{M_{r,d}^L, (E,\phi)})_w = \mathbb{H}^0(C, \mathbb{T}_{(E,\phi),w})
\end{equation}
which can be nonzero only for $-k\le w \le k+1$. Denote by $r_i, d_i$ the rank and degree of $E_i$, respectively.
Note that if $r_i=1$ for all $i$, then $(E,\phi)$ must be generically regular.
Indeed, under the BNR correspondence, the sheaf it corresponds to restricts to the structure sheaf of $rC$ over the locus where all $E_i$ trivialize.

\begin{proof}[Proof of \Cref{prop:twistedstrata}]\label{pf:twistedstrata}
    First, note that $\chi(\mathbb{T}_{(E,\phi),w})$ equals
    \[-\sum_{i-j=w}(r_id_j - r_jd_i + r_ir_j(1-g)) + \sum_{i-j=w}(r_id_{j+1} - r_{j+1}d_i + r_ir_{j+1}(l + 1-g)),\]
    where $l=\deg L$.
    In particular, we obtain:
    \begin{align*}
        \sum_{w=-k}^0 \chi(\mathbb{T}_{(E,\phi),w}) &= -\sum_{i}(r_i^2(1-g)) + \sum_{i<j} r_ir_jl\\
        &= \frac{l}2r^2 - \left(\frac{l}2+1-g\right)\sum_ir_i^2.
    \end{align*}

    Since $(E,\phi)$ is stable, we know that
    \[\mathbb{H}^{-1}(C, \mathbb{T}) = \mathbb{H}^0(C, \mathcal H^{-1}(\mathbb{T})) \]
    is one-dimensional and has pure weight $0$.
    On the other hand,
    \[\dim\mathbb H^1(C, \mathbb{T}) = \dim\mathbb H^1(C, \mathcal H^0(\mathbb{T})) =\begin{cases} 1 & L = K_C\\ 0 & \deg L> 2g-2\end{cases}\]
    and $\mathbb H^1(C, \mathcal H^0(\mathbb{T}))$ has pure weight $1$.
    So if $(E,\phi)\in Z_m$, then we can compute
    \[\dim T_{S_m, (E,\phi)} = \sum_{w=-k}^0\dim \mathbb{H}^0(C, \mathbb{T}_{(E,\phi),w}) = \frac{l}2r^2 - \left(\frac{l}2+1-g\right)\sum_ir_i^2 + 1.\]

    When $l >2g-2$, this is bounded above by
    \[\binom{r}{2}l + r(g-1) + 1,\]
    and this bound is obtained when $r_i=1$ for all $i$.
    The very same number is also the relative dimension of $h_L$.
    Therefore $\dim S_m=\dim h_L^{-1}(0)$ precisely when $r_i=1$ for all $i$.
\end{proof}

The weights for the tangent spaces also provide another obstruction to the reducedness of irregular components:
\begin{corollary}\label{cor:hitchinirregnonred}
    When $L = K_C$, all irregular components are non-reduced.
\end{corollary}

\begin{proof}
    First, note that $h_L^{-1}(0)$ is lci and cut out by functions 
    \[(f_{i,j}),1\le i\le r,1\le j \le h^0(L^{\otimes i})\]
    pulled back from $B$, where we may pick $f_{ij}$ to be homogeneous functions for the $\mathbb{G}_m$-action of weight $-i$. 
    
    Fix a point $z\in (M_{r,d}^L)^{\mathbb{G}_m}$, lying in a stratum $S$, with eigenspace decomposition as in \Cref{eq:bundledecomposition} such that there is some $i$ with $r_i\neq 1$.
    Let $R = \mathcal O_{M_{r,d}^L, z}$; this admits a $\mathbb{G}_m$-action, such that the ideal for $S$, which we denote by $I_S\subset R$, is cut out by the ideal generated by all homogeneous functions with negative weights.
    The weights for generators of such an ideal are given by the weight decomposition of the conormal
    \[N_{S_m/M_{r,d}^L, z}^{\vee} = (T_{M_{r,d}^L, z}^{>0})^{\vee}.\]
    In particular, the lowest $\mathbb{G}_m$-weight is $-k-1$.
    
    It follows that any homogeneous function $f\in R$ of weight at most $-k-2$ lies in $I_S^2$.
    In particular, since $k<r-1$, the $f_{r, j}$ all lie in $I_S^2$, and the non-reducedness of $h_L^{-1}(0)$ along $S$ reduces to the lemma below. 
\end{proof}

\begin{lemma}
    Let $X$ be a smooth scheme and $Z\subset Z'\subset X$ subschemes of the same dimension.
    Assume that both $Z$ and $Z'$ are local complete intersections in $X$, and that $Z'$ is cut out by the regular sequence $f_1',\dots, f_k'$.
    Suppose further that $Z$ is irreducible, and cut out by the ideal $I_Z$. If any of the $f_i'$ lie in $I_Z^2$, then the multiplicity of $Z'$ along $Z$ is at least $2$. 
\end{lemma}
\begin{proof}
    Suppose $f_1'\in I_Z^2$. Then the chain of ideals $I_Z\supsetneq I_Z^2 + (f_2', \dots, f_k')\supset (f_1', \dots, f_k')$ shows the result.  
\end{proof}

\begin{remark}
    In \cite[Theorem 5.2]{HH}, it is shown how to compute the multiplicity of any such component which admits a very stable Higgs bundle.
    The arguments above can be seen as an extension of the techniques used there, which allows us to access multiplicities even when there is no very stable Higgs bundle. 
\end{remark}

\section{Multiplicities for generically regular components}\label{sec:gregmult}

\subsection{Stable Pairs}

Let $S$ be a smooth quasiprojective surface, and $\beta\in H_2(S;\mathbb Z)$ a curve class.
As in \Cref{ss:moduliofsheaves}, we denote by $M_S(\beta,\chi)$ the moduli of Gieseker-stable pure one-dimensional sheaves on $S$ with proper Fitting support in $\beta$ and Euler characteristic $\chi$.

Like before, we let $|\beta|$ be the Hilbert scheme of curves in $S$ of class $\beta$, and $\mathcal C_{\beta}\to |\beta|$ the universal family of curves.
Let $g_{\beta}$ be its genus.
Denote by
\[h:M_S(\beta,\chi) \rightarrow |\beta|\]
the map sending a sheaf $\mathcal F$ to its Fitting support.

Tensoring with the polarization does not change the geometry of this fibration, so one may feel free to increase $\chi$ without changing $(\chi\bmod\deg\beta)$ until it is large enough for our purpose.

As in the formulation in \cite{PT}, there is a moduli space $P_S(\beta, \chi)$ of stable pairs on $S$.
These are pairs $(\mathcal F,s)$ where $\mathcal F$ is a pure sheaf with the prescribed numerical condition, and $s\in\Gamma(\mathcal F)$ has zero-dimensional cokernel.
\begin{theorem}[{\cite[Proposition B.8]{PT}}]
    If $\chi \ge 1-g_\beta$ where $g_\beta$ is the genus of $\beta$, then $P_S(\beta,\chi)$ is isomorphic to the relative Hilbert scheme $\operatorname{Hilb}^{\chi+g_\beta-1}_{\mathcal C_\beta/|\beta|}$.
\end{theorem}
Going forward, we shall set $n=\chi+g_\beta-1$.
When we say something is true for $n\gg0$ or $\chi\gg 0$, we will mean that it is true after replacing $\chi$ with a higher value without changing $(\chi\bmod\deg\beta)$.

Let $\mathring{P_S}(\beta,\chi) \subset P_S(\beta,\chi)$ be the open subscheme of stable pairs whose underlying sheaf is also stable.
There is a forgetful map $f: \mathring P_S(\beta,\chi)\rightarrow M_S(\beta,\chi)$.
\begin{proposition}\label{prop:paircover}
    For $\chi\gg0$, the forgetful map $f$ is smooth and geometrically connected, and its image contains all stable sheaves whose scheme-theoretic support and Fitting support agree.
\end{proposition}
\begin{proof}
    Assume, without loss of generality, that $\chi$ is large enough such that all stable sheaves of class $\beta$ and Euler characteristic $\chi$ are globally generated and have vanishing higher cohomologies.

    By construction $\mathring P_S(\beta,\chi)$ is an open in a projective bundle over $M_S(\beta,\chi)$.
    Indeed, if $\mathcal F^{\rm univ}$ is the universal sheaf on $S\times M_{S}(\beta,\chi)$, then its pushforward to this moduli space is a vector bundle, the projectivization of which contains $\mathring P_S(\beta,\chi)$ as an open.

    Now suppose $\mathcal F\in\operatorname{Coh}(S)$ is any globally-generated pure one-dimensional sheaf whose scheme-theoretic support and Fitting support agree.
    Denote its support by $\widetilde C\hookrightarrow S$.
    Then there is a section $\mathcal O_{\widetilde C}\to\mathcal  F$ which is injective over a dense open of $\widetilde{C}$.
    The cokernel must have dimension $0$ as $\widetilde{C}$ is also the Fitting support, so this gives a stable pair.
\end{proof}

\begin{corollary}\label{cor:hilbcover}
    Let $C$ be a curve of class $\beta$, corresponding to a point $b\in |\beta|$. Then for $n\gg0$, there is an open $\mathring{\operatorname{Hilb}}^n_C\subset \operatorname{Hilb}^n_C$ and a smooth map $f:\mathring{\operatorname{Hilb}}^n_C \rightarrow h^{-1}(b)$.
\end{corollary}
\subsection{Nowhere-reducedness of Hilbert schemes}
We say a scheme is nowhere reduced if its smooth locus is empty.
\begin{lemma}\label{lem:hilbetale}
    Suppose $C,C'$ are smooth curves on smooth quasiprojective surfaces $S,S'$.
    Then $\operatorname{Hilb}^n_{rC}$ is nowhere reduced if and only if $\operatorname{Hilb}^n_{rC'}$ is.
\end{lemma}
In fact, the proof will show that the lists of multiplicities appearing in these Hilbert schemes coincide.
\begin{proof}
    We work in the analytic category, thus identifying Hilbert schemes as Douady spaces.
    The relevant references can be found in \cite{dCM_douady}.

    Suppose $Z\in\operatorname{Hilb}_{rC}^n$ is a finite subscheme.
    Take a small neighborhood $U$ of $Z^{\rm red}$ in $S$.
    Then $\operatorname{Hilb}_{rC\cap U}^n$ is a neighborhood of $Z$ in $\operatorname{Hilb}_{rC}^n$.

    The locus $\operatorname{Hilb}_{rC\cap U}^n\subset\operatorname{Hilb}_U^n$ is cut out by the condition $f^r|_Z=0$, where $Z$ is a length-$n$ subscheme of $U$, and $f$ is a local function whose vanishing defines $C$.

    Choose any $|Z^{\rm red}|$ many disjoint points on $C'$ and take an open neighborhood $U'$ of these points in $S'$.
    The characterization above then shows that, up to shrinking $U,U'$ further, we have $\operatorname{Hilb}_{rC\cap U}^n\cong\operatorname{Hilb}_{rC'\cap U'}^n$.

    Therefore the Hilbert schemes $\operatorname{Hilb}^n_{rC},\operatorname{Hilb}^n_{rC'}$ admit open covers with isomorphic members.
    The claim follows.
\end{proof}

So to compute the multiplicities of $\operatorname{Hilb}_{rC}^n$, it suffices to assume $C=\mathbb A^1$ is a line in $\mathbb A^2$.
In this case, \cite{luan} showed that the irreducible components of $\operatorname{Hilb}_{rC}^n$ are indexed by partitions of the form
\[n=m_1+\cdots+m_k,m_i\le r\]
where a generic point of the component corresponding to $m_1,\ldots,m_k$ is a subscheme of the form $Z=Z_1\sqcup\cdots\sqcup Z_k$ with $\operatorname{length}Z_i=m_i$.
\begin{theorem}[\cite{luan}]
    The multiplicity of the component corresponding to the partition $m_1,\ldots,m_k$ is at least
    \[
    \prod_i(r-m_i+1).
    \]
\end{theorem}
\begin{remark}
    The current version of \cite{luan} claims that the multiplicity is precisely the quoted value.
    This seems to be in error, as the proof only shows the statement above.
\end{remark}
\begin{corollary}\label{greg_nonred}
    Any generically regular component of the global nilpotent cone $h^{-1}(rC)$ is non-reduced if either $r>2$ or $r=2$ and $2\mid\deg_CN_{C/S}$.
\end{corollary}
\begin{proof}
    Combining \Cref{lem:hilbetale} with \Cref{cor:hilbcover}, it suffices to check that the numbers in the preceding theorem can never all be $1$ in our situation.

    Recall that
    \[g_\beta-1\equiv \binom{r}{2}\deg_CN_{C/S}\pmod{r}.\]
    So if either $r>2$ or $r=2$ and $2\mid\deg_CN_{C/S}$, then $\gcd(g_\beta-1,r)>1$.

    As $\gcd(r,\chi)=1$, $n=\chi+g_\beta-1$ is not divisible by $r$ (even after possibly replacing $\chi$ with a higher value, for it does not change the residue class modulo $\deg\beta=\deg r[C]=r\deg [C]$).
    So it cannot be that $m_i=r$ for all $i$.
\end{proof}
In particular,
\begin{corollary}
    Suppose $\deg L>2g-2$.
    Then the global nilpotent cone for the $L$-twisted Hitchin fibration is nowhere reduced if either $r>2$ or $r=2$ and $2\mid \deg L$.
\end{corollary}

\section{Del Pezzo surfaces}
\subsection{Moduli of sheaves on a del Pezzo surface}
Now let $(S,H)$ be a polarized del Pezzo surface, \emph{i.e.}~a smooth projective surface with ample anticanonical bundle.
Recall that we have
\[H^1(S,\mathcal O)=H^2(S,\mathcal O)=0.\]

Let $C\subset S$ be a smooth curve of genus $g\ge2$, and $\chi\in\mathbb Z$ with $\gcd(C\cdot H,\chi)=1$.

Let $r>0$ be such that $\gcd(r,\chi)=1$.
Denote by $\beta$ the class of $rC$.
Note that in this case $M_S(\beta,\chi)$ is a smooth projective variety and the support fibration $h:M_S(\beta,\chi)\to |\beta|$ is flat \cite{yuan2023}.

We aim to show the following statement, which may be viewed as an analogue of \Cref{prop:twistedstrata}:
\begin{proposition}\label{prop:delpezzogreg}
    For $n\gg0$, there is an open subscheme $\mathring{\operatorname{Hilb}}^n_{rC}\subset\operatorname{Hilb}^n_{rC}$ such that the smooth map $\mathring{\operatorname{Hilb}}^n_{rC}\rightarrow h^{-1}(rC)$ is dominant. 
\end{proposition}
\subsection{Deformation to the normal cone}
We make use of the deformation to the normal cone construction for the inclusion $C\hookrightarrow S$.
It gives a family of surfaces
\[\mathcal S\rightarrow \mathbb{A}^1\]
with $\mathcal S_t=S$ for $t\neq 0$, and $\mathcal S_0 = \operatorname{Tot}_L$ where $L$ is the normal bundle to $C$.
Note that $\deg L=\deg(K_C-K_S|_C)>2g-2$.

Associated to this data is a family of homology classes $\beta_t=r[C]\in H_2(\mathcal S_t;\mathbb Z)$.

Let $\mathcal M\to\mathbb A^1$ be the relative moduli space of sheaves that parameterizes, for each $t\in\mathbb A^1$, pure sheaves with proper Fitting support in $\beta_t$ and Euler characteristic $\chi$.
Let $\mathcal B\to \mathbb A^1$ be the relative Hilbert scheme of curves with class $\beta_t$.
The support map then gives an $\mathbb A^1$-morphism
\[\begin{tikzcd}
    \mathcal M\arrow{rr}\drar&&\mathcal B\dlar\\
    &\mathbb A^1&
\end{tikzcd}.\]
\begin{lemma}
    The spaces $\mathcal M$ and $\mathcal B$ are irreducible varieties smooth over $\mathbb A^1$, and the map $\mathcal M\to\mathcal B$ is proper and flat.
\end{lemma}
\begin{proof}
    This is standard (see \cite{dCMS22}).
    We just need to show that $\dim\mathcal B_t$ is independent of $t$.
    This reduces to showing
    \begin{equation}\label{eq:dimeq_delpezzo}
        h^0(\mathcal O_S(rC))-1=\sum_{i=1}^rh^0(\mathcal O_C(iL)).
    \end{equation}
    Since $\deg L>2g-2$, we have $h^0(\mathcal O_C(iL))=\chi(\mathcal O_C(iL))$ for all $i>0$.
    Also, as $-K_S$ is ample, $H^2(\mathcal O_S(rC))=0$ for all $r\ge 0$.

    We do induction on $r$ on the following statement:
    The cohomology $H^1(\mathcal O_S(rC))$ vanishes and \Cref{eq:dimeq_delpezzo} holds.

    For $r=0$ this is clear by Kodaira vanishing.
    Assuming this holds for $r$, then we have
    \[\sum_{i=1}^{r+1}h^0(\mathcal O_C(iL))=\chi(\mathcal O_S(rC))-1+\chi(\mathcal O_C((r+1)L)).\]

    Consider now the exact sequence
    \[\begin{tikzcd}
        0\rar&\mathcal O_S(rC)\rar&\mathcal O_S((r+1)C)\rar&\mathcal O_C((r+1)L)\rar&0
    \end{tikzcd}\]
    which immediately shows that $H^1(\mathcal O_S((r+1)C))=0$ and \Cref{eq:dimeq_delpezzo} follows.
\end{proof}
The map $\mathcal B\to\mathbb A^1$ has an obvious section given by the nilpotent curve $rC$.
This gives a family $\mathcal C_{\rm nilp}\to\mathbb A^1$ of curves given by $rC$ and a family $\mathcal M_{\rm nilp}\to\mathbb A^1$ of the corresponding global nilpotent cone.
\begin{proof}[Proof of \Cref{prop:delpezzogreg}]
    Consider the Hilbert scheme fibration
    \[\operatorname{Hilb}_{\mathcal C_{\rm nilp}/\mathbb{A}^1}^n\rightarrow \mathbb{A}^1.\]

    By \Cref{prop:paircover}, possibly after increasing $\chi$ and $n$, there is a smooth cover
    \[\mathring{\operatorname{Hilb}}^n_{\mathcal C_{\rm nilp}/\mathbb{A}^1} \rightarrow \mathcal M_{\rm nilp}\]
    relative to $\mathbb A^1$.
    By \Cref{prop:twistedstrata}, the base change over $0\in \mathbb{A}^1$ is dominant.
    Then the statement follows from the lemma below.
\end{proof}

\begin{lemma}
    Let $f:Y\rightarrow X$ be a smooth map and $X\to B$ proper and equidimensional.
    If $f_b:Y_b\rightarrow X_b$ is dominant for some $b\in B$, then the same is true for all $b$ in some nonempty open $U\subset B$.
\end{lemma}

\begin{proof}
    Let $d$ be the dimension of the fibers of $X\rightarrow B$.
    Let $U\subset B$ be the open where the map $X\setminus f(Y)\rightarrow B$ has fibers of dimension less than $d$; this is a nonempty neighborhood of $b\in B$ by assumption. The result follows by equidimensionality of $X\rightarrow B$. 
\end{proof}

\begin{corollary}
    Suppose either $r>2$ or $r=2$ and $2\mid \deg_CN_{C/S}$.
    Then the global nilpotent cone $h^{-1}(rC)$ is nowhere reduced.
\end{corollary}

\section{Primitivity of fibers}
\subsection{The setup}
Suppose $S$ is a K3 surface, $\ell$ a primitive, basepoint-free, big and nef class on $S$, and $r,\chi$ coprime integers.
Denote by
\[M=M_S(r\ell,\chi)\to B=|r\ell|\]
the associated Beauville-Mukai system.
The purpose of this section is to prove the following:
\begin{proposition}\label{prop:primfibers}
    The integral homology class of a general fiber of $M\to B$ is primitive (that is, indivisible) if and only if $r=1$.
\end{proposition}
\subsection{A deformation to Hilbert scheme}
We show that the integral homology class is primitive if $r=1$.
The argument first appeared in a previous version of \cite{bai2026}.

The technique is to deform any such Beauville-Mukai system to a situation with a section.
Take an elliptic K3 surface $\rho:S_0\to\mathbb P^1$ with integral fibers and a section $s$, such that $\operatorname{Pic}(S_0)=\mathbb Zc_1(s)\oplus\mathbb Zf$ where $f\in H^2(S_0,\mathbb Z)$ is the class of a general fiber.

There is a fibration $\rho^{[g]}:S_0^{[g]}\to(\mathbb P^1)^{(g)}=\mathbb P^g$ given by composing the Hilbert-Chow map $S_0^{[g]}\to S_0^{(g)}$ with the functorial map on the symmetric product $S_0^{(g)}\to (\mathbb P^1)^{(g)}$.

Let $\ell_0=s+gf$.
Then we can naturally identify $(\mathbb P^1)^{(g)}\cong |\ell_0|$ by sending $p_1+\cdots+p_g\in (\mathbb P^1)^{(g)}$ to the effective Cartier divisor $s+\rho^{-1}(p_1)+\cdots+\rho^{-1}(p_g)$.
\begin{proposition}
    There is an isomorphism $M_{S_0}(\ell_0,\chi)\cong S_0^{[g]}$ such that the diagram
    \[\begin{tikzcd}
        M_{S_0}(\ell_0,\chi)\dar \rar{\cong} & S_0^{[g]} \dar\\
        {|\ell_0|} \rar{\cong} & (\mathbb P^1)^{(g)}
    \end{tikzcd}\]
    commutes.
\end{proposition}
\begin{proof}
    Let us first construct this for $\chi>1$.
    This is a consequence of \cite[Theorem 3.15]{yoshioka}, which says that the relative Fourier-Mukai transform on the elliptic surface gives an isomorphism
    \[M_{S_0}(1,(\chi-1)f,1-g)\cong M_{S_0}(\ell_0,\chi)\]
    for any $\chi>1$.
    Here, $M_{S_0}(1,(\chi-1)f,1-g)$ is the moduli of stable sheaves on $S_0$ with Mukai vector $(1,(\chi-1)f,1-g)$.

    Note now that $M_{S_0}(1,(\chi-1)f,1-g)$ consists precisely of sheaves of the form $\mathcal I_Z\otimes\mathcal O((\chi-1)f)$ where $Z$ is a length $g$ subscheme.
    By the calculation in \cite[(3.15)]{yoshioka} (see also \cite{bai2025}), the image of such a sheaf under this isomorphism is supported on the Cartier divisor $s+\rho^{-1}(p_1)+\cdots+\rho^{-1}(p_g)$, where $p_1+\cdots+p_g=\rho^{[g]}(Z)$.

    Hence the stated isomorphism is the desired one.
    For $\chi\le 1$, simply note that the operation $-\otimes\mathcal O(\ell_0)$ establishes an isomorphism
    \[M_{S_0}(\ell_0,\chi)\to M_{S_0}(\ell_0,\chi+2g-2)\]
    that clearly preserves the respective support maps.
\end{proof}
\begin{corollary}
    The Beauville-Mukai system $M_{S_0}(\ell_0,\chi)\to |\ell_0|$ has a section.
    In particular, the integral homology class of a general fiber is primitive.
\end{corollary}
There is a smooth connected Deligne-Mumford stack parameterizing pairs $(S,\ell)$ where $S$ is a K3 surface, and $\ell$ is a quasi-polarization, \emph{i.e.}~a big and nef line bundle on $S$ whose class is primitive and satisfies $\ell^2=2g-2$.
Let $\mathcal H_g$ be a cover of this moduli space by a smooth connected variety.
\begin{proof}[Proof of the ``if'' part in \Cref{prop:primfibers}]
    The construction of Beauville-Mukai systems is functorial, so we obtain a family of fibrations
    \[u:\mathcal M\to\mathcal H_g\times\mathbb P^g\]
    such that the fiber over any $(S,\ell)\in\mathcal H_g$ is the Beauville-Mukai system
    \[M_S(\ell,\chi)\to\mathbb P^g.\]

    We claim that if $(S,\ell)\in\mathcal H_g$ satisfies the condition that a general fiber of the system $M_S(\ell,\chi)\to\mathbb P^g$ has primitive homology, then so does any $(S',\ell')$ contained in a contractible open $U\subset\mathcal H_g$ containing $(S,\ell)$.
    Indeed, let us consider the restriction
    \[\mathcal M_U\to U\times\mathbb P^g\]
    as a fibration.
    All general fibers of this are then homologous in $H^{2g}(\mathcal M_U;\mathbb Z)$.

    On the other hand, for any $(S',\ell')\in U$, there is a natural isomorphism
    \[H^{2g}(M_{S'}(\ell',\chi);\mathbb Z)\cong H^{2g}(\mathcal M_U;\mathbb Z)\]
    as $U$ is contractible.
    But the class of a general fiber is primitive in $H^{2g}(M_S(\ell,\chi);\mathbb Z)$, so it would have to be primitive in $H^{2g}(M_{S'}(\ell',\chi);\mathbb Z)$ as well.

    As $\mathcal H_g$ can be covered by contractible opens, it remains to find one $(S,\ell)\in\mathcal H_g$ such that the conclusion of the proposition holds.
    This is given by the preceding corollary.
\end{proof}
\subsection{Multiplicities in generically abelian fibrations}
Let us now demonstrate that when $r>1$ the fiber class is not primitive.
The K3 condition is not necessary.
In fact, we will show:
\begin{proposition}\label{prop:nonprimfibers}
    Suppose we are under the hypotheses of \Cref{intro:mainthm} for either (CY1) or (F1) with $r>1$, then a general fiber of the support fibration is not primitive.
\end{proposition}
This implies the ``only if'' part in \Cref{prop:primfibers}.
Indeed, since $\ell$ is basepoint-free and big and nef, $|\ell|$ contains a smooth curve $C$ of genus at least $2$.

We make use of the following result:
\begin{theorem}[\cite{CKV,lu_preprint_01}]
    Suppose $M\to B$ is a generically abelian fibration, and $F$ is a fiber of it.
    If the greatest common divisor of multiplicities of the components of $F$ is $1$, then $F$ has a reduced component.
\end{theorem}
\begin{corollary}
    Suppose $h:M\to B$ is a generically abelian fibration with $B$ proper, and suppose it has a fiber with no reduced components, then a general fiber of $h$ does not have primitive integral homology class.
\end{corollary}
\begin{proof}
    For any point $p\in B$, the class of the fiber $M_p$ in $M$ coincides with the pullback $h^\ast[p]$ on cohomology.

    The integral homology class of $[M_p]$ is divisible by the greatest common divisor of the multiplicities of its components.
    By the preceding theorem, this greatest common divisor is at least $2$ if $M_p$ has no reduced components.

    On the other hand, all points in $B$ are homologous, so the corollary follows.
\end{proof}
\begin{proof}[Proof of \Cref{prop:nonprimfibers}]
    Suppose $r>1$.
    Then by what we know already, the fiber over $rC$ is nowhere reduced.
    On the other hand, the support fibration is generically abelian as its general fiber is a Jacobian of a smooth curve.
    The result follows from the preceding corollary.
\end{proof}

\bibliographystyle{plain}
\bibliography{ref}

\end{document}